\documentclass{article}
\usepackage{mathtools,amsthm,amssymb}
\usepackage{authblk}

\theoremstyle{theorem}
\newtheorem{theorem}{Theorem}
\newtheorem{lem}{Lemma}

\theoremstyle{definition}
\newtheorem{definition}{Definition}

\newcommand\extrafootertext[1]{%
    \bgroup
    \renewcommand\thefootnote{\fnsymbol{footnote}}%
    \renewcommand\thempfootnote{\fnsymbol{mpfootnote}}%
    \footnotetext[0]{#1}%
    \egroup
}

\begin{document}

\title{Factoring Variants of Chebyshev Polynomials\\with Minimal Polynomials of $\cos(\frac{2\pi}{d})$}
\markright{Factoring Variants of Chebyshev Polynomials}
\author{D.A.\ Wolfram}
\affil{\small College of Engineering \& Computer Science\\The Australian National University, Canberra, ACT 0200\\
\medskip
{\rm David.Wolfram@anu.edu.au}
}
\date{}

\maketitle

\begin{abstract}
We solve the problem  of factoring polynomials $V_n(x) \pm 1$ and $W_n(x) \pm 1$  where $V_n(x)$ and $W_n(x)$ are  Chebyshev polynomials of the third and fourth kinds. The method of proof is based on previous work by Wolfram~\cite{wolfram} for factoring variants of  Chebyshev polynomials of the first  and second  kinds,  $T_n(x) \pm 1$ and $U_n(x) \pm 1$. We also  show that, in general, there are no factorizations of variants of Chebyshev polynomials of the fifth and sixth kinds, $X_n(x) \pm 1$ and $Y_n(x) \pm 1$   using minimal polynomials of $\cos(\frac{2\pi}{d})$.
\end{abstract}

\extrafootertext{MSC: Primary 12E10, Secondary 12D05}

\section{Background.}

Chebyshev polynomials of the third and fourth kinds were named by Gautschi~\cite{gautschi} and are also called airfoil polynomials~\cite{fromme,handscomb}. They are used in areas such as  solving differential equations~\cite{Abd-Elhameed}, numerical integration~\cite{aghigh,fromme,handscomb}, approximations~\cite{handscomb}, interpolation~\cite{gautschi} and combinatorics~\cite{andrews}.  The significance of their applications in mathematics, engineering and numerical modeling provides a motivation for studying the properties of these polynomials. 

In previous work, Wolfram~\cite{wolfram} solved an open factorization problem for Chebyshev polynomials of the second kind $U_n(x) \pm 1$, and gave a more direct proof of the  result for Chebyshev polynomials of the first kind, $T_n(x) \pm 1$ . We apply this method to solve the analogous factorization problems for Chebyshev polynomials of the third and fourth kinds. All of these factorizations can be expressed in terms of the minimal polynomials of $\cos(\frac{2\pi}{d})$.

\subsection{Chebyshev Polynomials of the Second Kind}
Chebyshev polynomials of the second kind can be defined by 
\begin{equation} \label{Uprop}
U_n(x) = \frac{\sin((n + 1)\theta)}{\sin(\theta)}
\end{equation} where $x = \cos \theta$ and  $n \geq 0$. This is equation (1.4) of Mason and Handscomb~\cite{handscomb}. It follows that 
\begin{equation} \label{UTuran}
U_n(x)^2 -1 = U_{n-1}(x) U_{n+1}(x)
\end{equation} where $n \geq 1$, by applying the trigonometric identity  $\sin^2 A - \sin^2 B = \sin(A + B) \sin(A - B)$ with $A = (n+1)\theta$ and $B = \theta$.

These polynomials satisfy the following recurrence, for example equations (1.6a)--(1.6b) of \cite{handscomb}:
\begin{align} \label{Urec}
 U_0(x) &= 1, &   U_1(x) &= 2x, &U_{n}(x) = 2 x U_{n-1}(x) - U_{n-2}(x)
\end{align} where $n > 1$.

G{\"u}rta\c{s}~\cite{gurtas} showed that
\begin{equation} \label{UPsi}
U_{n-1}(x) =  \prod_{\mathclap{\substack{d \mid 2n \\d > 2}}} \Psi_d(x) 
\end{equation} where  $n \geq 1$.
The polynomials $\Psi_d(x)$ are 
\begin{equation} \label{defPsi}
\Psi_d(x) =  \prod_{\mathclap{\substack{k \in S_{d/2}}}} 2 \left(x - \cos \left(2 \pi \frac{k}{d}\right) \right)
\end{equation} where $S_{d/2} =\{k \mid (k, d) = 1, 1 \leq k < d/2\}$ and $d > 2$. They have degree $\phi(d)/2$ where $\phi$ is Euler's totient function~\cite{gurtas}.
The minimal polynomial in $\mathbb{Q}[x]$ of $\cos(\frac{2\pi}{d})$ is $2^{-\frac{\phi(d)}{2}}\Psi_d(x)$ where $d > 2$ which follows from the proof of Theorem 1 of D. H. Lehmer~\cite{lehmer}.

\begin{definition}  \label{def1}
The polynomial $\Psi_1(x) =2(x -1)$ and  $\Psi_2(x) = 2(x+1)$.  
\end{definition}

These are polynomials with roots $\cos(2 \pi)$  and of $\cos(\pi)$, respectively.  

\subsection{Chebyshev Polynomials of the Third and Fourth Kinds}

Chebyshev polynomials of the third kind can be defined by
\begin{equation} \label{VDef}
V_n(x) = \frac{\cos(n + \frac12)\theta}{\cos \frac{\theta}2}
\end{equation}
and of the fourth kind by
\begin{equation} \label{WDef}
W_n(x) = \frac{\sin(n + \frac12)\theta}{\sin \frac{\theta}2}
\end{equation} where $x = \cos \theta$ and $n \geq 0$.

They can also be defined with respect to Chebyshev polynomials of the second kind, $U_n(x)$:
\begin{equation} \label{VU}
V_n(x) = U_n(x) - U_{n-1}(x)
\end{equation} and
\begin{equation} \label{WU}
W_n(x) = U_n(x) + U_{n-1}(x)
\end{equation} where $n \geq 1$. These are equations (1.17)--(1.18) of Mason and Handscomb~\cite{handscomb}.

\section{Solution}

The method of solution follows that by Wolfram~\cite{wolfram}.The first step is to express $V_n(x)^2 - 1$ and $W_n(x)^2 -1$ in terms of the minimal polynomials $\Psi_d(x)$ where $d \geq 1$.

\begin{lem} \label{allfactors}
Where $n \geq 1$,
\begin{equation} \label{V2}
V_n(x)^2 -1 = \Psi_1(x)\prod_{\mathclap{\substack{d \mid 2n \\d > 2}}} \Psi_d(x) \prod_{\mathclap{\substack{d \mid 2n +2 \\d > 2}}} \Psi_d(x) 
\end{equation}
\begin{equation} \label{W2}
W_n(x)^2 -1 = \Psi_2(x)\prod_{\mathclap{\substack{d \mid 2n \\d > 2}}} \Psi_d(x) \prod_{\mathclap{\substack{d \mid 2n +2 \\d > 2}}} \Psi_d(x).
\end{equation}
\end{lem}

\begin{proof}
From equation (\ref{VU}), we have
\begin{align*}
V_n(x)^2 - 1 =& (U_n(x) - U_{n-1}(x) + 1) (U_n(x) - U_{n-1}(x) - 1) \mbox{ from equation (\ref{VU})}\\
=& U_n(x)^2 -1  -2 U_n(x) U_{n-1}(x)  + U_{n-1}(x)^2\\
=& U_{n+1}(x)U_{n-1}(x)  -2 U_n(x) U_{n-1}(x)  + U_{n-1}(x)^2 \mbox{ from equation (\ref{UTuran})}\\
=& U_{n-1}(x) (U_{n+1}(x) - 2 U_n(x)   + U_{n-1}(x))\\
=& U_{n-1}(x) (2 x U_n(x)   - 2 U_n(x) ) \mbox{ from equation (\ref{Urec})}\\
=& \Psi_1(x) U_{n-1}(x) U_n(x)  \mbox{ from Definition \ref{def1}}\\
=&  \Psi_1(x) \prod_{\mathclap{\substack{d \mid 2n \\d > 2}}} \Psi_d(x) \prod_{\mathclap{\substack{d | 2n +2 \\d > 2}}} \Psi_d(x) \mbox{ from equation (\ref{UPsi}).}
\end{align*}  Similarly, we have
\begin{align*}
W_n(x)^2 - 1 =& (U_n(x) + U_{n-1}(x) + 1) (U_n(x) + U_{n-1}(x) - 1) \mbox{ from equation (\ref{WU})}\\
=& U_n(x)^2 -1  +2 U_n(x) U_{n-1}(x)  + U_{n-1}(x)^2\\
=& U_{n+1}(x)U_{n-1}(x)  +2 U_n(x) U_{n-1}(x)  + U_{n-1}(x)^2 \mbox{ from equation (\ref{UTuran})}\\
=& U_{n-1}(x) (U_{n+1}(x) + 2 U_n(x)   + U_{n-1}(x))\\
=& U_{n-1}(x) (2 x U_n(x)   + 2 U_n(x) ) \mbox{ from equation (\ref{Urec})}\\
=& \Psi_2(x) U_{n-1}(x) U_n(x)  \mbox{ from Definition \ref{def1}}\\
=&   \Psi_2(x) \prod_{\mathclap{\substack{d \mid 2n \\d > 2}}} \Psi_d(x) \prod_{\mathclap{\substack{d | 2n +2 \\d > 2}}} \Psi_d(x) \mbox{ from equation (\ref{UPsi})}
\end{align*} as required. 
\end{proof}

The following theorem solves the factorization problem for $V_n(x)^2 -1$.  The second step of the method concerns defining the mapping that splits the $2n$ factors of $V_n(x)^2 -1$ into the $n$ factors of $V_n(x) +1$ and the other $n$ factors of $V_n(x) -1$. The factorizations are unique up to associativity and commutativity of multiplication.

\begin{theorem} \label{solutionV}
If  $n  \geq 1$,
\begin{equation} \label{Vsolplus}
V_n(x) + 1 = \prod_{\mathclap{\substack{d \mid 2n \\d > 2\\ 2n/d \:\mbox{\rm\small odd}}}} \Psi_d(x) \ \ \ \ \prod_{\mathclap{\substack{d \mid  2n + 2 \\d > 2\\ (2n + 2) / d \:\mbox{\rm\small odd}}}} \Psi_d(x)
\end{equation}

and
\begin{equation} \label{Vsolminus}
V_n(x) - 1 = \Psi_1(x)  \prod_{\mathclap{\substack{d \mid 2n \\d > 2\\ 2n/d \:\mbox{\rm\small even}}}} \Psi_d(x) \ \ \ \  \prod_{\mathclap{\substack{d \mid 2n + 2 \\d > 2\\ (2n + 2) / d \:\mbox{\rm\small even}}}} \Psi_d(x).
\end{equation}

\end{theorem}

\begin{proof}
The polynomial $\Psi_1(x) = 2 (x -1)$ is a factor of $V_n(x)^2 -1$ from equation~(\ref{V2}), and $\Psi_1(\cos(2 \pi)) = 0$. It follows from equation~(\ref{VDef}) that $V_n(\cos(2 \pi)) = 1$ and so $\Psi_1(x)$ is a factor of $V_n(x) -1$.

 If $d \mid 2n$ and $d > 2$,  let $\theta = \frac{2\pi k}{d}$ where $(k, d) = 1$, $1 \leq k < \frac{d}2$ and  $a = \frac{2n}{d}$. We have  $\theta = \frac{\pi a k }{n}$ and $\Psi_d(\cos(\theta)) = 0$.
 From equation~(\ref{VDef}), 
 \begin{align*}
 V_n(\cos(\theta)) =& \frac{\cos((n + \frac12) \frac{\pi a k}n )}{\cos(\frac{\theta}2)}\\
 =& \frac{\cos(\pi a k) \cos(\frac{\theta}2) - \sin(\pi a k) \sin(\frac{\theta}2)}{\cos(\frac{\theta}2)}.
 \end{align*}
The denominator $\cos(\frac{\theta}2) \not= 0 $ because $\frac{\theta}2 = \frac{\pi k}d$ cannot equal $\frac{\pi}2$ when $d > 2$.
The numbers $a k$ and $a$ have the same parity. This is immediate when $a$ is even. If $a$ is odd,  it follows that $d$ is even and $k$ is odd because $(k, d) = 1$. We have $\cos(\pi a k) = \cos(\pi a)$, and  $V_n(\cos(\theta)) = \cos(\pi a)$. 

Hence, if $a$ is even,  then $V_n(\cos(\theta)) = 1$ and $\Psi_d(x)$ is a factor of $V_n(x) - 1$. Similarly, if $a$ is odd, then $V_n(\cos(\theta)) = -1$ and $\Psi_d(x)$ is a factor of $V_n(x) + 1$. 
 \medskip
 
If $d \mid 2n+ 2$ and $d > 2$,   let  $b= \frac{2n + 2}{d}$. We have $\theta = \frac{\pi b k}{n+ 1}$ where $k$ is such that $(k, d) = 1$ and $1 \leq k < \frac{d}2$.
 From equation~(\ref{VDef}), 
 \begin{align*}
 V_n(\cos(\theta)) =& \frac{\cos((n + \frac12) \theta )}{\cos(\frac{\theta}2)}\\
 =& \frac{\cos((n +1 ) \theta - \frac{\theta}2)}{\cos(\frac{\theta}2)}\\
 =& \frac{\cos(\pi b k) \cos(\frac{\theta}2) + \sin(\pi b k) \sin(\frac{\theta}2)}{\cos(\frac{\theta}2)}.
 \end{align*}

Similarly to the previous case  the denominator $\cos(\frac{\theta}2) \not= 0$, the numbers $b k$ and $b$ have the same parity, and $V_n(\cos(\theta)) = \cos(\pi b)$. Ift follows that if $b$ is odd then $\Psi_d(x)$ is a factor of $V_n(x) +1$ and  if $b$ is even then $\Psi_d(x)$ is a factor of $V_n(x) -1$.

From equations~(\ref{Urec}) and~(\ref{VU}), $V_n(x)$ has degree $n$. It follows that the right side of equation~(\ref{V2}) of the factorization of $V_n(x)^2 -1$ has degree $2n$. It has $2n$ factors of the form $2(x - \cos(\theta))$ from  equation~(\ref{defPsi}) and Definition~\ref{def1}, half of which are the factors of $V_n(x) + 1$ and the other half are the factors of $V_n(x) -1$. The mapping defined above maps every such factor of $V_n(x)^2 - 1$ to either  $V_n(x) + 1$ or $V_n(x) - 1$ depending on whether $\cos(\theta)$ is either a root of  $V_n(x) + 1$ or $V_n(x) - 1$, respectively. The right sides of equations~(\ref{Vsolplus}) and (\ref{Vsolminus}) are the products of these mapped factors and so both have degree equal to $n$.

From equations (\ref{Urec}) and (\ref{VU}) and, the  leading coefficients of $V_n(x) \pm 1$ are $2^n$. The expansions of the  factorizations on the right sides of equations (\ref{Vsolplus}) and (\ref{Vsolminus}) both have $2^n$ as leading coefficients also, because each is a  product of $n$ factors of the form $2(x - \cos(\theta))$, as required. 
\end{proof}

\section{Examples with $V$}
The polynomial $V_{12}(x)^2 - 1$ has 24 factors, and $V_{12}(x) + 1$ and $V_{12}(x) -1$ each are the products of half of these factors. 
The mapping in the proof of Theorem \ref{solutionV} gives
\begin{align*}
V_{12}(x) + 1 =& \Psi_8(x) \Psi_{24}(x) \Psi_{26}(x)\\
V_{12}(x) - 1 =& \Psi_1(x) \Psi_3(x)\Psi_4(x)\Psi_6(x)\Psi_{12}(x)\Psi_{13}(x)\\
=& (2(x -1)) (2x + 1) (2x)  (2x -1)  (4 x^2 - 3)\\
&(64 x^6 + 32 x^5 - 80x^4 - 32 x^3 + 24 x^2 + 6x -1)
\end{align*}

The following theorem solves the factorization problem for $W_n(x)^2 -1$. These factorizations are also unique up to associativity and commutativity of multiplication.
\begin{theorem} \label{solutionW}
If  $n  \geq 1$,
\begin{equation} \label{Wsolplus}
W_n(x) + 1 =   \prod_{\mathclap{\substack{d \mid 2n \\d > 1\\ 2n/d \:\mbox{\rm\small odd}}}} \Psi_d(x) \ \ \ \ \prod_{\mathclap{\substack{d \mid 2n + 2 \\d > 2\\ (2n + 2) / d \:\mbox{\rm\small even}}}} \Psi_d(x)
\end{equation}

and
\begin{equation} \label{Wsolminus}
W_n(x) - 1 = \prod_{\mathclap{\substack{d \mid 2n \\d > 1\\ 2n/d \:\mbox{\rm\small even}}}} \Psi_d(x) \ \ \ \  \prod_{\mathclap{\substack{d \mid 2n + 2 \\d > 2\\ (2n + 2) / d \:\mbox{\rm\small odd}}}} \Psi_d(x).
\end{equation}

\end{theorem}

\begin{proof}

The structure of the proof is similar to that of Theorem~\ref{solutionV}.
If $d \mid 2n$ and $d > 1$, let $a = \frac{2n}{d}$ and $k$ be such that $(k, d) = 1$ where $1 \leq k < \frac{d}2$. From equation~(\ref{WDef}), 
 \begin{align*}
 W_n(\cos(\theta)) =& \frac{\sin((n + \frac12) \frac{\pi a k}n )}{\sin(\frac{\theta}2)}\\
 =& \frac{\cos(\pi a k) \sin(\frac{\theta}2) + \sin(\pi a k) \cos(\frac{\theta}2)}{\sin(\frac{\theta}2)}.
 \end{align*}
 
The denominator $\sin(\frac{\theta}2) \not= 0 $ because $\frac{\theta}2 = \frac{\pi k}d$ cannot equal $\pi$ when $d > 1$. Similarly, we have that  $a k$ and $a$ have the same parity, and  $W_n(\cos(\theta)) = \cos(\pi a)$. Hence, if $a$ is even,  then $W_n(\cos(\theta)) = 1$ and $\Psi_d(x)$ is a factor of $W_n(x) - 1$. If $a$ is odd, then $W_n(\cos(\theta)) = -1$ and $\Psi_d(x)$ is a factor of $W_n(x) + 1$. 
\medskip
 
 If $d \mid 2n+ 2$ and $d > 2$,   let  $b= \frac{2n + 2}{d}$ and $k$ be such that $(k, d) = 1$ where $1 \leq k < \frac{d}2$. We have $\theta = \frac{\pi b k}{n+ 1}$ and $\frac{\theta}2 = \frac{\pi k}d$.
 From equation~(\ref{WDef}), 
 \begin{align*}
 W_n(\cos(\theta)) =& \frac{\sin((n + \frac12) \theta )}{\sin(\frac{\theta}2)}\\
 =& \frac{\sin((n +1 ) \theta - \frac{\theta}2)}{\sin(\frac{\theta}2)}\\
 =& \frac{-\cos(\pi b k) \sin(\frac{\theta}2) + \sin(\pi b k) \cos(\frac{\theta}2)}{\sin(\frac{\theta}2)}.
 \end{align*}
 The denominator $\sin(\frac{\theta}2) \not= 0 $, and $b$ and $bk$ have the same parity, as above.
Hence, if $b$ is even,  then $W_n(\cos(\theta)) = -1$ and $\Psi_d(x)$ is a factor of $W_n(x) + 1$. If $b$ is odd, then $W_n(\cos(\theta)) = 1$ and $\Psi_d(x)$ is a factor of $W_n(x) - 1$.

It is straightforward to show that the degrees of the right sides of equations~(\ref{Wsolplus}) and (\ref{Wsolminus}) are both $n$, and the leading coefficients of both sides of these equations is $2^n$.
\end{proof}

\section{Examples with $W$}
The polynomial $W_{12}(x)^2 - 1$ has 24 factors, and $W_{12}(x) + 1$ and $W_{12}(x) -1$ each are the products of half of these factors. 
The mapping in the proof of Theorem \ref{solutionW} gives
\begin{align*}
W_{12}(x) + 1 =&  \Psi_8(x) \Psi_{24}(x) \Psi_{13}(x)\\
W_{12}(x) - 1 =& \Psi_2(x) \Psi_3(x)\Psi_4(x)\Psi_6(x)\Psi_{12}(x)\Psi_{26}(x)\\
=& (2(x+1))(2x + 1) (2x)  (2x -1)  (4 x^2 - 3)\\
&(64 x^6 - 32 x^5 - 80x^4 
+ 32 x^3 + 24 x^2 - 6x -1)
\end{align*}

When $n$ is odd, $\Psi_2(x)$ is a factor of $W_n(x) +1$:
\begin{align*}
W_{11}(x) + 1 =&  \Psi_2(x) \Psi_{22}(x) \Psi_{3}(x) \Psi_{4}(x)\Psi_6(x) \Psi_{12}(x).
\end{align*}

\section{Chebyshev Polynomials of the Fifth and Sixth Kinds}

Chebyshev polynomials of the fifth kind, $X_n(x)$, and sixth kind, $Y_n(x)$, were defined by  Masjed-Jamei~\cite{XY}. Similarly to the other four kinds of Chebyshev polynomials, they are orthogonal polynomials with integer coefficients and $X_n(x)$ and $Y_n(x)$ have degree $n$ where $n \geq 0$~\cite{Youssri,Youssri2}. They have the form 
\[
\sum_{v=0}^{\lfloor \frac{n}2 \rfloor} a_v x^{n - 2v}.
\]

Monic Chebyshev polynomials of the fifth and sixth kinds, $\bar{X}_n(x)$ and $\bar{Y}_n(x)$, can be defined by the following recurrences  which we simplify from~\cite{Youssri}.
\begin{align}
G_{0,m}(x) =& 1 \nonumber\\
G_{1, m}(x)=& x \nonumber\\
G_{n, m}(x) =& x G_{n -1, m}(x) + A_{n-1, m} \; G_{n-2, m}(x),\; \;  \mbox{ $n > 1$}
\end{align}
where
\begin{align}
A_{n, m} =& \frac{ (2n + m -2)(-1)^n + (2n -(m -2))  - nm -n^2}{(2 n + m -1)(2n +m -3)}, \;\; \mbox{ and}\\
\bar{X}_n(x) =& G_{n, 3}(x) \label{X}\\
\bar{Y}_n(x) =& G_{n, 5}(x). \label{Y}
\end{align}

The first six Chebyshev polynomials of the fifth kind over $\mathbb{Z}$  are
\begin{align*}
X_0(x) =& 1\\
X_1(x) =& x\\
X_2(x) =& 4x^2 - 3\\
X_3(x) =& 6x^3 - 5x\\
X_4(x) =& 16x^4 - 20x^2 + 5\\
X_5(x) =& 80x^5 - 112 x^3 + 35 x\\
X_6(x) =& 64x^6 - 112x^4 + 56x^2 -7.
\end{align*}

The first six Chebyshev polynomials of the sixth kind over $\mathbb{Z}$  are
\begin{align*}
Y_0(x) =& 1\\
Y_1(x) =& x\\
Y_2(x) =& 2x^2 - 1\\
Y_3(x) =& 8x^3 - 5x\\
Y_4(x) =& 16x^4 - 16x^2 + 3\\
Y_5(x) =& 24x^5 - 28 x^3 + 7 x\\
Y_6(x) =& 16x^6 - 24 x^4 + 10x^2 -1.
\end{align*}

The polynomials $X_5(x) \pm 1$ and $Y_5(x) \pm 1$ are irreducible over $\mathbb{Z}$ and none is a minimal polynomial of $\cos(\frac{2 \pi}{d})$, of which there are only two of degree $5$:
\begin{align*}
\Psi_{11}(x) =& 32 x^5 + 16 x^4 - 32 x^3 - 12 x^2 + 6x + 1\\
\Psi_{22}(x) =&32 x^5 - 16 x^4 - 32 x^3 + 12 x^2 + 6x - 1.
\end{align*}

These counter examples for $X_n(x) \pm 1$ and $Y_n(x) \pm 1$, show that, in general, they do not  have factorizations using the minimal polynomials of $\cos(\frac{2\pi}{d})$.

\section{Conclusion}
We have solved the problem of factoring variants of Chebyshev polynomials of the third and fourth kinds, $V_n(x) \pm 1$ and  $W_n(x) \pm 1$,  in terms of minimal polynomials for $\cos(\frac{2 \pi}{d})$. This was done by applying the method of Wolfram~\cite{wolfram} for  factoring $T_n(x) \pm 1$ and  $U_n(x) \pm 1$ in a similar way. We have shown that there are no generalizations of this factorization to similar variants of Chebyshev polynomials of the fifth and sixth kinds, $X_n(x) \pm 1$ and $Y_n(x) \pm 1$.

\section*{Acknowledgment}
 I am grateful  to the College of Engineering \& Computer Science at The Australian National University for research support.

\vfill\eject

\end{document}